\documentclass[12pt]{amsart}

\usepackage{enumerate}

\usepackage[active]{srcltx}
\usepackage{nicefrac}

\usepackage{amsthm,amssymb,setspace}
\usepackage[pagebackref,colorlinks,linkcolor=red,citecolor=blue,urlcolor=blue,hypertexnames=true]{hyperref}
\usepackage{amsrefs}
\usepackage[matrix, arrow]{xy}

\renewcommand{\deg}{{\rm degree}}

\newcommand{\N}{\mathbb N}
\newcommand{\R}{\mathbb R}

\newcommand{\C}{\mathbb C}

\theoremstyle{plain}
\newtheorem*{theorem*}{Theorem}
\newtheorem{theorem}{Theorem}[section]
\newtheorem{corollary}[theorem]{Corollary}
\newtheorem{lemma}[theorem]{Lemma}
\newtheorem{proposition}[theorem]{Proposition}

\theoremstyle{definition}
\newtheorem*{definition*}{Definition}
\newtheorem{definition}[theorem]{Definition}

\theoremstyle{remark}

\newtheorem{remark}[theorem]{Remark}
\newtheorem{example}[theorem]{Example}

\begin{document}

\onehalfspace

\title{Tracial algebras and an embedding theorem}

\author{Tim Netzer}
\address{Tim Netzer, Universit\"at Leipzig, Germany}
\email{netzer@math.uni-leipzig.de}

\author{Andreas Thom}
\address{Andreas Thom, Universit\"at Leipzig, Germany}
\email{thom@math.uni-leipzig.de}

\begin{abstract} We prove that every positive trace on a countably generated $*$-algebra can be approximated by positive traces on algebras of generic matrices. This implies that every countably generated tracial $*$-algebra can be embedded into a metric ultraproduct of generic matrix algebras. As a particular consequence, every finite von Neumann algebra with separable pre-dual can be embedded into an ultraproduct of tracial $*$-algebras, which as $*$-algebras embed into a matrix-ring over a commutative algebra.
\end{abstract}

\maketitle

\tableofcontents

\section*{Introduction}

The theory of \emph{Rings of operators} was founded by F.\ Murray and J.\ von Neumann in the first half of the last century, see \cites{mvn1,mvn2,mvn3}. Later, the term \emph{von Neumann algebra} was coined to emphasize the seminal contributions of John von Neumann. We will freely use standard results in the theory of von Neumann algebras. For those results and all necessary definitions we refer to \cite{tak2} and the references therein. An early achievement of Murray and von Neumann was a classification of von Neumann algebras into types. In this note we are concerned with \emph{finite} von Neumann algebras with separable pre-dual, which decompose as products of algebras of the types I$_n$ for $1 \leq n < \infty$ and type II$_1$. The classification of algebras of type I$_n$ is complete as they are all isomorphic to algebras $L^{\infty}(X,\C) \otimes M_n \C$ for some probability measure space $(X,\mu)$. Hence, the key objects of study are II$_1$-factors with separable pre-dual, i.e. infinite-dimensional and weakly closed $*$-subalgebras of $B(H)$, which carry a faithful trace, have trivial center and admit a countable weakly dense subset. There are various constructions of II$_1$-factors, e.g. from groups or group actions, but also via internal constructions like free products, amplification etc. An important question right from the beginning was to which extend general II$_1$-factors are close to matrix algebras. A main result in the work of Murray and von Neumann was the construction of the hyperfinite II$_1$-factor $R$ and the proof of its uniqueness with respect to some \emph{strong} form of approximation with matrices. Murray and von Neumann also gave examples of II$_1$-factors, i.e. free group factors, which were not hyperfinite.

The embedding conjecture of Alain Connes states that every type II$_1$-factor with separable pre-dual can be embedded into an ultraproduct of the hyperfinite II$_1$-factor.
This assertion is equivalent some \emph{weak} form of approximation by matrices which ought to hold always.
The conjecture dates back to Connes' seminal work on injective von Neumann algebras,
\cite[p.\ 105]{connes1}. Although it is well-known that many II$_1$-factors, including free group factors, do embed into an ultraproduct of the hyperfinite II$_1$-factor, this conjecture remains open and has triggered a lot of interesting and deepgoing research.
There are various ways of reformulating the Connes embedding problem, and one way of putting it is to ask for an embedding into an metric ultraproduct of a sequence of finite von Neumann algebras of type I. We will prove that such an embedding always exists if one allows to approximate with a more general class of tracial $*$-algebras of type I.

Motivated by the work of D.\ Hadwin, see \cite{hadwin}, F.\ R\u adulescu established a relationship between the Connes embedding conjecture and some analogue of Hilbert's 17th problem on positive polynomials, see \cite{radulescu}. 
This approach was carried further in a more algebraic 
setup by I.\ Klep and M.\ Schweighofer, see \cite{klepsch}. In this note we want to follow this line of approach, and instead of giving yet another reformulation of the original problem we will obtain an affirmative answer to a different but analogous problem. Indeed, we will enlarge the realm and consider general $*$-algebras with a positive, faithful and unital trace. More precisely:

\begin{definition*}
A tracial $*$-algebra is a unital complex algebra $A$ with involution and a complex-linear functional $\tau \colon A \to \C$ such that:
\begin{enumerate}
\item $\tau(1) = 1$,
\item $\tau(a^*a)\geq 0$ for all $a\in A$, and
\item $\tau(ab) = \tau(ba)$, for all $a,b \in A$.
\end{enumerate}
\end{definition*}

$\tau$ is called a (positive) trace.
Note that $\tau(a^*)=\overline{\tau(a)}$ for all $a\in A$ is an automatic consequence, since $\tau\left((1+a^*)(1+a)\right)\geq 0 \mbox{ and }  \tau\left((1-ia^*)(1+ia)\right)\geq 0.$

In the context of this definition, the trace $\tau$ is said to be \emph{faithful} if $\tau(a^*a)=0$ holds only if $a=0$; the tracial $*$-algebra $(A,\tau)$ is called \emph{trace-reduced} then. We say that $a \in A$ is \emph{$\tau$-bounded}, if $\tau((a^*a)^p) \leq C^{2p}$ for some constant $C>0$, and all $p\in\N$. Moreover, $(A,\tau)$ is called \emph{bounded} if every element in $A$ is $\tau$-bounded. It is clear that every finite von Neumann algebra with a specified trace gives rise to a bounded and trace-reduced tracial $*$-algebra. Conversely, the GNS-construction allows to construct a trace-preserving embedding of any bounded and trace-reduced tracial $*$-algebra into a finite von Neumann algebra (this is for example demonstrated in \cite{pa}). The lack of bounded elements in a general tracial $*$-algebra causes many pathologies and gives rise to phenomena that makes the study of general tracial $*$-algebras quite different compared to the study of finite von Neumann algebras.

However, in this more general setup, we may still talk about algebras of various types. Indeed, a finite von Neumann algebra is a sum of algebras of the form $L^{\infty}(X_k,\C) \otimes_{\C} M_k \C$ for $k \leq n$ if and only if it satisfies a certain universal polynomial identity, similar to the commutator relation which characterizes type I$_1$.  Hence, we will say that a tracial $*$-algebra is of type I$_{\leq n}$ if it satisfies this identity (Definition \ref{deftyp}). Equivalently, we could say that a tracial $*$-algebra is of type I$_{\leq n}$ if and only if it embeds (not necessarily preserving the unit) into a ring of $n\times n$-matrices over a commutative $\C$-algebra. Algebras which satisfy a polynomial identity are called PI-algebras and have been studied in detail over the last decades. We will recall various results in Section \ref{polid} and refer for the general theory to \cite{row}.

We are going to define a suitable notion of metric ultraproduct for arbitrary tracial $*$-algebras -- which contains the usual ultraproduct of von Neumann algebras -- and prove a general embedding theorem. Indeed, the main application of our results is the following theorem (Corollary \ref{coroneumann} combined with Definition \ref{deftyp} and Corollary \ref{embed}):

\begin{theorem*}
Let $(M,\tau)$ be a finite von Neumann algebra with separable pre-dual. Then there exists a sequence $(A_n,\tau_n)$ of trace-reduced tracial $*$-algebras such that 
\begin{enumerate}
\item For every $n \in \N$, the complex algebra $A_n$ embeds as a complex algebra into the ring of $n\times n$-matrices over a commutative complex algebra, and
\item there exists a trace-preserving embedding
$$\iota \colon (M,\tau) \hookrightarrow \prod_{n \to \omega} (A_n,\tau_n).$$
\end{enumerate}
\end{theorem*}

Our results (Theorems \ref{main} and \ref{mainc}) cover more generally all countably generated trace-reduced tracial $*$-algebras, but the case of finite von Neumann algebras is certainly the most interesting one.

\section{Preliminaries} \label{prel}
\subsection{Convex geometry}
We first recall some basic concepts and results from convex geometry. See for example \cite{barvinok} for details and proofs.

\begin{definition} Let $V$ be a real vector space.
A subset $C \subset V$ is said to be a \emph{convex cone}, if
$$\quad C+C \subset C, \quad \mbox{and} \quad \R_{\geq 0} \cdot C \subset C.$$
\end{definition}

A linear functional $\phi \colon V \to \R$ is said to be positive with respect to the cone $C$, if $\phi(c) \geq 0$ for all $c \in C$. We will frequently apply this to the situation where $C$ contains a linear subspace $L$, where we get $L\subseteq C \cap (-C) \subset \ker \phi$.



For any subset  $C\subseteq V$ we denote by $C^\vee$ its \emph{dual}, which is by definition:
$$C^\vee = \{\phi\colon V \to \R \mbox{ linear }\mid \phi(c) \geq 0, \forall c \in C\}.$$ The double dual of $C$ (in $V$) is $$C^{\vee\vee} = \{ v\in V\mid \phi(v)\geq 0, \forall \phi\in C^\vee\}.$$
A classical result following from the Hahn-Banach Theorem is
$$C^{\vee\vee} = {\rm \overline{conv}}(C),$$
i.e.\ the double dual is nothing but the closed convex hull of $C$, in the finest locally convex topology on $V$.
Another useful property is that
$$(A+B)^\vee = A^\vee \cap B^\vee.$$

We record some facts about the finest locally convex topology. It can be defined as the coarsest vector space topology making all seminorms continuous.  All subspaces are then closed and inherit again the finest locally convex topology. So every finite dimensional subspace inherits the euclidean topology. All linear mappings into a vector space with any locally convex topology are continuous.
The following result is Proposition 1 in \cite{bisgaard}:

\begin{lemma} \label{technical}
A subset $B \subset  V$  in a countable dimensional real vector space is closed with respect to the finest locally convex topology if and only if
$B \cap W $ is closed for all finite dimensional subspaces $W$ of $V$.
\end{lemma}

A somewhat subtle consequence, which we can derive is the following result:

\begin{proposition} \label{rie}
Let $V$ be a countable dimensional real vector space, equipped with the finest locally  convex topology, and let $F \subset V$ be 
a finite dimensional subspace, with a fixed norm. Let $B \subset V$ be a closed convex cone, and let $\varepsilon >0$ be arbitrary.
If $\phi \colon F \to \R$ is a linear functional which satisfies $\phi(b) \geq 0$, for all $b \in F \cap B$, then there exists a linear functional
$\psi \colon V \to \R$ with $\psi(b) \geq 0$, for all $b \in B$, and $\|\psi|_F- \phi\|\leq \varepsilon$.
\end{proposition}
\begin{proof} First assume that also $V$ is finite dimensional. Let $F \subset V$ have codimension $k$.
Let $V'$ be the dual space of $V$ and $F^\vee$ the $k$-dimensional linear space of functionals which vanish on $F$. Denote by $B^{\vee}$ the dual cone of $B$. The space of all extensions of $\phi$ is an affine $k$-plane of the form $\phi + F^\vee$. Our assumption on the positivity of $\phi$ translates simply into 
\begin{equation} \label{eqpos}
\phi + F^\vee \subset \overline{B^\vee + F^{\vee}}.
\end{equation}
Indeed the dual cone of $F \cap B$ is just the closure of $B^\vee + F^\vee$, and every linear extension of $\phi$ will be positive for the cone $F \cap B$. Note, at this point we are using that $B$ is a closed convex cone.

We see from (\ref{eqpos}) that every open neighborhood of $\phi + F^\vee$ has to intersect $B^\vee + F^\vee$. Hence we can perturb $\phi$ as little as we wish to some functional $\phi'$, and we can produce some element $\psi$ in the intersection of $\phi' + F^\vee$ and $B^\vee$, hence a positive extension of $\phi'$.

Now let $V$ be countable dimensional and let $(V_i)_{i\geq 1}$ be an increasing sequence of finite dimensional subspaces, with $V_0:=F\subseteq V_i$ for all $i\geq 1$, and $V=\bigcup_i V_i$. Equip the $V_i$ with compatible norms. Now choose inductively for each $i$ a linear functional $\psi_i\colon V_i\rightarrow \R$ with $\psi_i\geq 0$ on $B\cap V_i$ and $$\Vert \psi_i \vert_{V_{i-1}} -\psi_{i-1}\Vert \leq \frac{\varepsilon}{2^i}.$$ This can be done by the already proven result for finite dimension. Then for all $0\leq j< i$ $$\Vert\psi_i\vert_{V_{j}}-\psi_j\Vert \leq \sum_{k=j+1}^i \frac{\varepsilon}{2^i}.$$ This shows that the sequence $(\psi_i)_i$ converges pointwise on $V$ to some linear functional $\psi\colon V\rightarrow \R$, which is nonnegative on $B$ and fulfills $\Vert \psi\vert_F -\phi\Vert\leq\varepsilon$.
\end{proof}


\subsection{The algebra of non-commutative polynomials}

Throughout $X$ denotes the finite set of letters $X_1,\dots,X_n$. We denote by $\C\langle X \rangle$ the algebra of non-commutative polynomials in $X$ with
coefficients in $\C$, i.e. $\C\langle X\rangle$ consists of all $\C$-linear combinations of words $\omega$ in the letters $X$. A monomial is an expression $\lambda \omega$ with a word $\omega$ and $\lambda\in\C$.  Each monomial in $\C\langle X\rangle$ has a natural length which we call its \emph{degree}. In this way the algebra $\C\langle X\rangle$ becomes a graded algebra; the degree of a polynomial is the highest degree of a monomial occurring with a non-zero coefficient. 

$\C\langle X \rangle$ is equipped with an involution $P \mapsto P^{*}$ such that $X_i^*=X_i$ for all $i$ and $\lambda^*=\overline{\lambda}$ for $\lambda\in\C$. The set of hermitian elements $$\C \langle X\rangle^h:=\left\{ P\in \C\langle X\rangle\mid P^*=P\right\}$$ carries the structure of a real vector space. It is also a homogeneous subset of $\C\langle X\rangle$, i.e. a polynomial is hermitian if and only if all of its homogeneous parts are hermitian. 
We will always equip this space with the finest locally convex topology.

We denote by $\C\langle X\rangle _{\leq k}$ and $\C\langle X\rangle_{\leq k}^h$ the subspaces of elements of degree at most $k$.

\begin{definition} We denote by $Q(\C\langle X \rangle)$ the convex cone of sums of hermitian squares:
$$Q(\C\langle X \rangle) = \left\{ \sum_{i=1}^m P_i^* P_i \mid m \in \N, P_i \in \C\langle X \rangle \right\}.$$
Clearly, $Q(\C\langle X\rangle)\subseteq \C\langle X\rangle^h$.
\end{definition}

A linear functional is said to be \emph{positive} if it is positive with respect to the convex cone of hermitian squares.

\begin{definition}
Denote by $C_{\rm cyc} \subset \C\langle X \rangle$ the $\C$-subspace which is spanned by all 
linear commutators. Two elements $P_1,P_2 \in \C\langle X \rangle$ are said to be \emph{cyclically equivalent}
if their difference is in $C_{\rm cyc}$, i.e.\ it is a sum of commutators. Write $C_{\rm cyc}^h:=C_{\rm cyc}\cap \C\langle X\rangle^h.$
\end{definition}

\begin{remark} \label{hom}Note that $C_{\rm cyc}$ is a homogeneous subspace of $\C\langle X\rangle$, i.e. the homogeneous parts of a sum of commutators are again sums of commutators.
So $C_{\rm cyc}^h$ is also a homogeneous subspace of $\C\langle X\rangle^h.$
\end{remark}


\begin{lemma} \label{ext}We have $C_{\rm cyc}^h + i\cdot C_{\rm cyc}^h=C_{\rm cyc}.$ Thus   $C_{\rm cyc}^h$ is the real vector space spanned by all hermitian commutators. 
\end{lemma}
\begin{proof} Let $P,Q\in A:=\C\langle X\rangle$. Write $P=a+ib, Q=c+id$ with $a,b,c,d\in A^h$. Then \begin{align*} PQ-QP & = (a+ib)(c+id) - (c+id)(a+ib)\\ &= ac - ca + db-bd + i(bc-cb) + i(ad-da)\\ &=  i(bc-cb) + i(ad-da) +i ( -i(ac - ca) - i(db-bd)).  \end{align*} In the last line, all the elements $i(bc-cb),i(ad-da), -i(ac-ca)$ and $-i(db-bd)$ are hermitian commutators. This establishes the first claim. The second claim follows immediately from the first.
\end{proof}

\begin{remark} \label{commut}
Note that the intersection between between $Q(\C\langle X \rangle)$ and $C_{\rm cyc}$
consists only of $0$. There are many ways to see this; the most intuitive way is to say that identities in $\C\langle X \rangle$ can be checked by checking that they hold whenever one specifies to self-adjoint matrices. This is justified in Theorem \ref{infinite}. In matrices, a commutator has always vanishing trace whereas a sum of squares with vanishing trace is zero.
\end{remark}

The following lemma is usually attributed to C.\ Caratheodory, but has been rediscovered for example in the work of K.\ Schm\"udgen, see page 325 of
\cite{schm2}.

\begin{lemma} \label{cara}
Let $P \in \C\langle X \rangle_{\leq m} \cap Q(\C\langle X \rangle)$. Then, there exist elements $P_1,\dots,P_{(n+1)^m} \in \C\langle X \rangle_{\leq m/2}$,
such that
$$P= \sum_{i=1}^{(n+1)^m} P_i^*P_i.$$
\end{lemma}

K.\ Schm\"udgen has used this Lemma to show that $Q(\C\langle X\rangle)$ is closed, see Theorem $4.2$ in \cite{schm1}. We are going to generalize this result and we will see in Corollary \ref{by}  that also $Q(\C\langle X\rangle)+ C_{\rm cyc}^h$ is a closed convex cone in $\C\langle X\rangle^h.$ 
\begin{remark}\label{extr}Assume $\tau$ is a normalized $\R$-linear functional on $\C\langle X\rangle^h$ (normalized means $\tau(1)=1$), which is positive with respect to the convex cone $Q(\C\langle X\rangle)+ C_{\rm cyc}^h$.
Then the complex linear extension  on $\C\langle X\rangle$  is nothing but a positive trace.
Indeed, since $C_{\rm cyc}^h$ is a linear subspace on which $\tau$ is positive, it has to vanish on it. Hence the complex linear extension vanishes on $C_{\rm cyc}$, by Lemma \ref{ext}, and  is thus  a trace.
\end{remark}

We will later also be interested in the algebra generated by countably many non-commuting variables $X_1,X_2,\ldots$. We denote this algebra by $\C\langle X_{\infty} \rangle$. The involution and the grading extend naturally. We define all notions as $C_{\rm cyc, \infty}, C_{\rm cyc, \infty}^h,\ldots  $ in the obvious way for this bigger algebra. Note however that $\C\langle X_{\infty}\rangle_{\leq m}$ is not a finite dimensional space any more.


\subsection{Polynomial identities} \label{polid}

Recall that a complex algebra  $A$ \textit{satisfies a polynomial identity} if there is a nonzero polynomial $P\in\C\langle X_1,\ldots,X_n\rangle$, for some $n\in\N$, such that $P(a_1,\ldots,a_n)=0$ holds for all $a_1,\ldots,a_n\in A$.  Such an algebra $A$ is then called a \textit{PI-algebra}.  Note that if $A$ satisfies a polynomial identity, then it satisfies a polynomial identity in 2 variables already. Indeed, replacing $X_i$ by $U^iV$ in a polynomial identity for $A$  results in a nontrivial polynomial identity in the two variables $U,V$. 

Let us recall the following facts about PI-algebras. Finitely generated PI-algebras are Jacobson rings, i.e. every prime ideal is an intersection of maximal ideals (see \cite{ampro}, Corollary 1.3). The Jacobson radical of a finitely generated PI-algebra is nilpotent (see \cite{br}, Theorem B).
Every prime PI-ring is a subring of the $r\times r$-matrices, for some $r$, over a division ring which is finite dimensional over its center, and it has a two-sided ring of quotients which is all of the matrix ring. This is Posner's Theorem, see \cite{pos}. In particular, every prime PI-ring is a subring of the $r\times r$-matrices over some (commutative) field.

Denote by $J_k$ the two-sided $*$-ideal in $\C\langle X \rangle$ of all expressions which
are zero, whenever we specify the $X$ to self-adjoint complex $k\times k$-matrices. We let $J_k^h$ be the real vector space of self-adjoint elements in $J_k$. In the same way we define $J_{k,\infty}$ and $J_{k,\infty}^h$ in $\C\langle X_{\infty}\rangle$, but we will consider these sets only later.

\begin{proposition} \label{graded} Let $k \geq 1$. The ideal $J_k$ is homogeneous, i.e. if $P$ is in $J_k$, then
so is each of its homogenous parts.  So also $J_k^h$ is homogeneous in $\C\langle X\rangle^h$. Further, $J_k^h +i\cdot J_k^h=J_k$ holds.
\end{proposition}
\begin{proof}
We define the number operator $N \colon \C\langle X \rangle \to \C\langle X \rangle$ to be the linear extension of multiplication by $2^{\deg(w)}$ on a monomial $w$. Clearly,
$$P(2a_1,\dots,2a_n) = (NP)(a_1,\dots,a_n),$$
for self-adjoint matrices $a_1,\dots,a_n \in M_{k} \C$. Hence $J_k$ is closed under the number operator. It is easy to conclude from this that $J_k$ is indeed homogeneous. Then clearly $J_k^h$ is homogeneous in $\C\langle X\rangle^h$. Finally let $P\in J_k$. Write $P=P_1 +iP_2$ with $P_1,P_2\in \C\langle X\rangle^h$. It follows easily that $P_1,P_2\in J_k$, which completes the proof.
\end{proof}

\begin{proposition} \label{ideal}
Let $k \geq 1$. Let $P_1,\dots,P_m \in \C\langle X \rangle$. If
\begin{equation} \label{condi}
\sum_{i=1}^m P_i^*P_i \in J_k + C_{\rm cyc},
\end{equation}
then $P_i \in J_k$, for all $1 \leq i \leq m$.
\end{proposition}
\begin{proof}
Condition \ref{condi}
says that $$\sum_{i=1}^m{\rm tr}(P_i(a_1,\dots,a_n)^*P_i(a_1,\dots,a_n))=0$$ for all self-adjoint $k \times k$-matrices $a_1,\dots,a_n$. Since this is just the sum of the squares of the Hilbert-Schmidt norms we can conclude that $$P_i(a_1,\dots,a_n)=0.$$ Hence, $P_i \in J_k$ as desired.
\end{proof}

A priori, there is a little ambiguity in the definition of $J_k$, since we could have taken all complex matrices instead of self-adjoint matrices. The following lemma shows that this makes to difference:

\begin{lemma} \label{complexreal}
Let $P \in J_k$ and $a_1,\dots,a_n$ be complex $k \times k$-matrices. Then,
$$P(a_1,\dots,a_n)=0.$$
\end{lemma}
\begin{proof}
We first claim that $P$ can be derived from a multi-linear relation with the same degree but in possibly more variables. Using a similar argument as in the proof of Proposition \ref{graded}, we may assume that $P$ is homogenous of multi-degree $(k_1,\dots,k_n)$. We set $m= \sum_{i=1}^n k_i$ and let:
$$Q(x_1,\dots,x_{m}) = P(x_1 + \cdots +x_{k_1},x_{k_1+1} + \cdots + x_{k_1 + k_2}, \cdots, x_{m - k_n+1} + \cdots + x_m ).$$
Then $Q=0$ is an identity that holds for all self-adjoint matrices of size $k$. Let $Q'$ be its homogenous part of multi-degree $(1,\dots,1)$, which is also an identity for all self-adjoint matrices of size $k$. Clearly,
$$P(x_1,\dots,x_n) = \frac1{k_1! k_2! \cdots k_n!} Q'(\underbrace{x_1,\dots,x_1}_{k_1},\underbrace{x_2,\dots,x_2}_{k_2}, \dots,\underbrace{x_n, \dots,x_n}_{k_n}),$$
and $Q'$ is multilinear. Just by linearity of $Q'$, the relation $Q'=0$ does also hold for complex matrices. Hence, $P=0$ holds for all complex matrices.
\end{proof}

\begin{remark}\label{eval}
We see from Lemma \ref{complexreal} that for $P\in J_k$ and $Q_1,\ldots,Q_n\in \C\langle X\rangle$, the polynomial $P(Q_1,\ldots,Q_n)$ again belongs to $J_k.$
\end{remark}

We define
\begin{equation*} \label{amitsur}
j_k = \sum_{\sigma \in S_{2k}} (-1)^{{\rm sg}(\sigma)} X_{\sigma(1)} X_{\sigma(2)} \cdots X_{\sigma(2k)}
\in\C\langle X_1,\dots,X_{2k} \rangle.
\end{equation*}

The most striking result about the ideal $J_k$ is a consequence of a Theorem of Amitsur and Levitzki, which can be found in \cite{amitlev1,amitlev2}.

\begin{theorem}[Amitsur-Levitzki] \label{infinite}
Any polynomial identity which holds in the ring of complex 
$k \times k$-matrices has degree $\geq 2k$. 
Moreover, every such polynomial identity of degree $2k$ is a scalar multiple of $j_k$.
\end{theorem}

The first useful consequence of the Amitsur-Levitzki theorem to our set of 
questions is the following corollary, which will be crucial in
the proof of Theorem \ref{crucial}.

\begin{corollary} 
Any homogeneous part of an element in $J_k$ has degree $\geq 2k$.
\end{corollary}
\begin{proof}
Since $J_k$ is homogeneous, any homogeneous part of an element from $J_k$ again belongs to $J_k$.  Using Lemma \ref{complexreal}, the corresponding relation holds for all complex matrices. By Theorem \ref{infinite}, no such element can exist of degree $< 2k$. This finishes the proof.
\end{proof}

Since $M_k(\C)$ satisfies the identity $j_k$, it satisfies also an identity in two variables. So all $J_k$ are nontrivial ideals in $\C\langle X\rangle$, for $n\geq 2$. The quotient algebra $A_{n,k}:=\C\langle X\rangle/J_k$ is a PI-algebra, fulfilling every identity from $J_k$. This can for example be seen from Remark \ref{eval}. We call this quotient the \textit{algebra of $n$ generic self-adjoint $k\times k$-matrices}. 

This notion can be justified as follows.  Let $\xi_{ij}^{(l)}$ be commuting variables, for $i,j=1\ldots k$ and $l=1,\ldots, n$. Let $C_{n,k}=\C[\xi_{ij}^{(k)}\mid i,j=1,\ldots,k, l=1,\ldots,n]$ be the commutative algebra generated by these variables, and extend the involution from $\C$ by defining  $(\xi_{ij}^{(k)})^*=\xi_{ji}^{(k)}$. Now let $M_k(C_{n,k})$ be the algebra of $k\times k$-matrices  over $C_{n,k}$, equipped with the canonical extended involution. For $l=1,\ldots,n$ write $Y_l= (\xi_{ij}^{(l)})_{i,j}$ and let $G_{n,k}$ be the subalgebra of $M_k(C_{n,k})$ generated by $Y_1,\ldots, Y_n$.  Now one checks that sending $X_i$ to $Y_i$ induces a well-defined $*$-algebra isomorphism between $A_{n,k}$ and $G_{n,k}$.

There is a large amount of literature on PI-algebras and especially  algebras of generic matrices, see for example \cite{jac, row}. Note however that algebras of generic matrices are often defined to be quotients with respect to the ideal $J_{k,\infty}$  of the polynomial algebra $\C\langle X_{\infty}\rangle$ in \textit{countably} many variables.  So we denote by $A_k$ the quotient $\C\langle X_{\infty}\rangle / J_{k,\infty}$ and call it the \textit{algebra of generic self-adjoint $k\times k$-matrices}.  Clearly each $A_{n,k}$ is a subalgebra of $A_k$.

The algebras $A_k$ are very nice structures. For example they are domains (see for example \cite{jac} Section II.4.), and by Posner's Theorem \cite{pos} they are subrings of matrices over a division ring that is finite dimensional over its center. The same is thus true for the algebras $A_{n,k}$.
We will use all these algebras below   to approximate an arbitrary finitely generated tracial $*$-algebra.

We end this section with a definition of \emph{type} in the realm of tracial $*$-algebras. We have argued in the Introduction that this extends the usual classification of finite von Neumann algebras into types.

\begin{definition} \label{deftyp}
A tracial $*$-algebra is said to be of \emph{type} I$_{\leq k}$ if it satisfies the polynomial identity 
\begin{equation*}
j_k = \sum_{\sigma \in S_{2k}} (-1)^{{\rm sg}(\sigma)} X_{\sigma(1)} X_{\sigma(2)} \cdots X_{\sigma(2k)}
\in\C\langle X_1,\dots,X_{2k} \rangle.
\end{equation*}
\end{definition}

\section{The main theorems} 

\label{secmain}

\subsection{Approximation of traces}

\label{secapp}

We will deduce our main results as a consequence of rather elementary results in convex geometry, for example Proposition \ref{rie}, and the main work is needed to show that those results  can be applied. Therefore, we will spend some time to show that certain cones in $\C\langle X \rangle$ under consideration are closed in the finest locally convex topology.

\begin{theorem} \label{endlich}
$Q(\C\langle X\rangle)+C_{\rm cyc}^h + J_k^h$ is a closed convex cone in $\C\langle X\rangle^h$.
\end{theorem}
\begin{proof} We have to show that  $\C\langle X\rangle_{\leq m}^h\cap \left( Q(\C\langle X\rangle)+C_{\rm cyc}^h + J_k^h\right)$ is closed, for all $m$. First take $a\in Q(\C\langle X\rangle)+C_{\rm cyc}^h + J_k^h$ of degree at most  $m$.
We write: $a = q + c + j$ with $q \in Q(\C\langle X \rangle), c \in C_{\rm cyc}^h$ and
$j \in J_k^h$. Assume $\deg(q)=2l >m$, write $q=\sum_i P_i^*P_i$ and let $Q_i$ denote the degree $l$ homogeneous part of $P_i$. Then $q'=\sum_i Q_i^*Q_i$ is the highest degree part of $q$, and from Proposition \ref{ideal} we see $Q_i\in J_k$ for all $i$. Note that we use homogeneity of $C_{\rm cyc}^h$ and $J_k^h$ at this point. So we can replace $P_i$ by $P_i-Q_i$, and add the counterterms $P_i^*Q_i + Q_i^*P_i -Q_i^*Q_i\in J_k^h$ to the element $j$. This  shows that we can assume $\deg(q)\leq m$ in the representation of $a$. Again by homogeneity of $C_{\rm cyc}^h $ and $J_k^h$ we can thus also assume $\deg(c)\leq m, \deg(j)\leq m$.

Now take any $y$ from the closure of $\C\langle X\rangle_{\leq m}^h\cap \left( Q(\C\langle X\rangle)+C_{\rm cyc}^h + J_k^h\right)$. There exist a sequence $(y_i)_{i \in \N}$ in 
$\C\langle X \rangle_{\leq m}^h \cap \left(Q(\C\langle X \rangle) + C_{\rm cyc}^h + J^h_k\right)$ with limit point $y \in \C\langle X \rangle^h_{\leq m}$.
We write $y_i = q_i + c_i + j_i$ as above, where the degrees of $q_i,c_i$ and $j_i$ do not exceed $m$. 

We claim that, up to a modification, we can choose a convergent subsequence of the sequence $(q_i)_{i\in N}$.
Indeed, let us consider a rapidly decreasing, nowhere vanishing and positive function $\psi$ on
$Y_k=(M_{k} \C^h)^n$, the $n$-fold cartesian product of the space of self-adjoint $k \times k$-matrices. We may assume that $$\int_{Y_k} \psi(x) d\lambda(x) =1.$$ Obviously, the kernel of the tautological homomorphism $\phi_k \colon \C \langle X \rangle \to C(Y_k,M_k\C)$ is precisely $J_k$. We define the $\psi$-normalized trace
$$\tau(P) = \int_{Y_k} {\rm tr}(\phi_k (P(x))) \psi(x) d\lambda(x),$$
which clearly vanishes on $C_{\rm cyc} + J_k$.
By Lemma \ref{cara}, we can write 
$$q_i = \sum_{j=1}^{(n+1)^{m}} x_{ij}^* x_{ij},$$ for some sequences
$(x_{ij})_{i \in \N}$ in $\C\langle X \rangle_{\leq m/2}$ and $1 \leq j \leq (n+1)^{m}$. We get 
$$\tau\left(\sum_{j} x_{ij}^* x_{ij}\right) \to \tau(y), \qquad \mbox{for } i \to \infty$$ by continuity. 
Hence, the sequences $(x_{ij})_{i \in \N}$ are bounded with respect
to the associated 2-norm on $\C\langle X \rangle_{\leq m/2}/\C\langle X \rangle_{\leq m/2} \cap J_k$. Therefore, we can lift them to
bounded sequences $(x_{ij} + y_{ij})_{i \in \N}$ in $\C\langle X \rangle_{\leq m/2}$ with $y_{ij} \in \C\langle X \rangle_{\leq m/2 }\cap J_k$. After a modification
in $\C \langle X \rangle_{\leq m} \cap J_k$, i.e.\ replacing $q_i$ by $\sum_{j} (x_{ij} + y_{ij})^*(x_{ij} + y_{ij})$, $j_i$ by $j_i - \sum_j \left( x_{ij}^*y_{ij} + y_{ij}^*x_{ij} + y_{ij}^*y_{ij}\right)$ and passing to a subsequence, we may thus assume that $q_i$ converges to some $q \in \C\langle X \rangle_{\leq m}$. Since $Q(\C\langle X\rangle)$ is closed, we get $q \in Q(\C\langle X\rangle) \cap \C\langle X \rangle_{\leq m}$.
Note, that now also $c_i + j_i \to c + j$ for some $c \in C_{\rm cyc}^h$ and $j \in J_k^h$,  since $C_{\rm cyc}^h + J_k^h$ is closed.
Hence $y = q + c + j  \in Q(\C\langle X\rangle)+C_{\rm cyc}^h + J_k^h,$ as desired.
\end{proof}

\begin{theorem} \label{crucial}
Let $2k >m$ be natural numbers. Then 
$$ \C\langle X \rangle_{\leq m}^h \cap \left(Q(\C\langle X \rangle) + C_{\rm cyc}^h +  J_k^h\right)
= \C\langle X \rangle_{\leq m}^h \cap \left(Q(\C\langle X \rangle) + C_{\rm cyc}^h\right)$$
\end{theorem}
\begin{proof}
The inclusion of the right side in the left side is obvious. For the other direction let $a \in Q(\C\langle X \rangle) + C_{\rm cyc}^h + J_k^h$ be  of degree at most  $m$.
Write $a = q + c + j$ with $q \in Q(\C\langle X \rangle), c \in C_{\rm cyc}^h$ and
$j \in J_k^h$. As in the proof of Theorem \ref{endlich} we can assume  that $q,c,j$ have degree at most $m$.
By Corollary \ref{infinite}, for $2k >m$, the least non-vanishing coefficient 
of $j$ has degree strictly bigger than $m$. Hence $j=0$ and the proof is finished.
\end{proof}

\begin{corollary}\label{by} $Q(\C\langle X\rangle)+ C_{\rm cyc}^h$ is a closed convex cone in $\C\langle X\rangle^h$.
\end{corollary}
\begin{proof}
Clear from Theorem \ref{endlich}, Theorem \ref{crucial} and Lemma \ref{technical}.
\end{proof}

\begin{theorem}\label{main1}
Let $\tau$ be a positive trace on $\C\langle X\rangle.$ Then there is a sequence of positive traces $(\tau_k)_{k\in\N}$, converging to $\tau$ pointwise on $\C\langle X\rangle$, such that each $\tau_k$ vanishes on $J_k$.
\end{theorem}
\begin{proof} Consider $\tau$ as an $\R$-linear functional on $\C\langle X\rangle_{\leq k}^h.$ It is clearly nonnegative on the convex cone $\left(Q(\C\langle X\rangle) + C_{\rm cyc}^h\right)\cap \C\langle X\rangle_{\leq k}.$ We can apply  Theorem \ref{crucial} and see that $\tau$ is nonnegative on $\left(Q(\C\langle X\rangle) + C_{\rm cyc}^h + J_k^h\right)\cap \C\langle X\rangle_{\leq k}.$ Now we apply Proposition \ref{rie} and find a linear functional $\tau_k$ on $\C\langle X\rangle^h$ that is nonnegative on $Q(\C\langle X\rangle) + C_{\rm cyc}^h + J_k^h$ and coincides on $\C\langle X\rangle_{\leq k}^h$ with $\tau$ up to $\frac1k$ say, in the operator norm. Note that we have also used Theorem \ref{endlich} here. The sequence of complex linear extension of the $\tau_k$ now converges pointwise on $\C\langle X\rangle$ to $\tau$. So in particular $\tau_k(1)\rightarrow \tau(1)=1$, and we can scale each $\tau_k$ with the positive factor $\frac{1}{\tau_k(1)}$ without destroying the convergence. The sequence we obtain in this way consists of positive traces as desired, see Remark \ref{extr}.
\end{proof}

\begin{remark} \label{infset}
Our setup is such that $X$ is always a finite set. 
If one considers a countable set of letters as we did already, a straightforward diagonalization procedure shows that the preceding theorem extends to traces on $\C\langle X_{\infty} \rangle$. Indeed one can to approximate the trace $\tau$ on the subalgebra $\C\langle X_1,\ldots,X_n\rangle$ by a trace $\tau_n,$ and use the canonical projection $\C\langle X_{\infty}\rangle \rightarrow \C\langle X_1,\ldots,X_n\rangle$ to pull $\tau_n$ back to a trace on $\C\langle X_{\infty}\rangle.$
\end{remark}

\subsection{H\"older and Minkowski type inequalities in tracial $*$-algebras}

Let $A$ be a unital $*$-algebra over $\C$.  Recall that a positive trace (or just a trace) on $A$ is a $\C$-linear mapping  $\tau\colon A\rightarrow \C$ fulfilling the following conditions: \begin{itemize}\item[(1)] $\tau(1)=1,$  \item[(2)] $ \tau(a^*a) \geq 0 $ for all $ a\in A,$ and \item[(3)] $\tau(ab)=\tau(ba)$ for all $a,b\in A$.
\end{itemize} 

Now for a fixed trace $\tau$, an even number $p$ and $a\in A$ we define $$ \Vert a\Vert_p := \tau\left( (a^*a)^{\frac{p}{2}}\right)^{\frac1p}.$$ This establishes a well defined mapping $\Vert\cdot\Vert_p\colon A\rightarrow \R_{\geq 0}$, which clearly fulfills $\Vert \lambda a\Vert_p = \vert\lambda\vert\cdot \Vert a\Vert_p$ and $\Vert a\Vert_p=\Vert a^*\Vert_p$   for $a\in A$ and $\lambda\in\C$.
In \cite{fackkos}, T.\ Fack and H.\ Kosaki showed that in a finite von Neumann algebra with a specified trace, one has
$$\|a + b\|_p \leq \|a\|_p + \|b\|_p, \quad \mbox{and} \quad \|ab\|_p \leq \|a\|_r \|b\|_s$$ for all $a,b \in A$, and positive $p,r,s$ satisfying
$\frac1p = \frac1r + \frac1s.$ The first inequality is a generalization of the classical Minkowski inequality whereas the second inequality generalizes H\"older's inequality. It is natural to expect that similar inequalities hold in the context of tracial $*$-algebras. We were not able to obtain the precise analogues of those inequalities, but we can establish the following weaker result.

\begin{theorem} \label{ineq}
Let $(A,\tau)$ be tracial $*$-algebra and let $p$ be an even positive integer.
We have,
$$\|a+b\|_p \leq \|a\|_q + \|b\|_q, \quad \mbox{and} \quad \|ab\|_p \leq \|a\|_{q'} \|b\|_{q'}$$
with $q=2^{2^{\frac p4-1}+1}$ and $q'= 2^{\frac{p}2+1}$.
\end{theorem}

The preceding theorem is a consequence of Proposition \ref{holder} and \ref{mink}, which we are going to prove in the sequel. We do not claim that the dependence of $q$ and $q'$ on $p$ is optimal. Indeed, we were not able to exclude the possibility that the inequalities of Minkowski and H\"older have an extension in their classical form to the realm of tracial $*$-algebras. One can easily see that for any trace which is a point-wise limit of bounded traces, the inequalities hold in their classical form.

\vspace{0.2cm}

For the whole section we will use the abbreviation $\varphi(n)=2^{\frac{n}{2}-1}$ for even $n\in\N$, and $\varphi(n)=\varphi(n+1)$ for $n$ odd.    We start with the following generalized Cauchy-Schwarz inequality: 

\begin{lemma}\label{cs} Let $\tau$ be a trace on $A$. For all $n\in\N$ and all  $a_1,\ldots,a_n\in A$ we have $$\vert \tau(a_1\cdots a_n)\vert\leq \prod_{i=1}^n \tau\left( (a_i^*a_i)^{\varphi(n)}  \right)^{\frac{1}{2\varphi(n)}}.$$ 
\end{lemma}

\begin{proof} The proof is by induction on $n$, first for $n$ even. For $n=2$ the statement follows immediately from the Cauchy-Schwarz inequality, using $\tau(a^*a)=\tau(aa^*)$ for all $a\in A$. For arbitrary even $n\geq 4$ we find \begin{align*}\vert\tau(a_1\cdots a_n)\vert & \leq \tau\left( a_{\frac{n}{2}}^{*} \cdots a_1^*a_1\cdots a_{\frac{n}{2}} \right)^{\frac12} \cdot \tau\left(a_n^* \cdots a_{\frac{n}{2}+1}^*a_{\frac{n}{2}+1}\cdots a_n\right)^{\frac12}\\ &= \tau\left(a_\frac{n}{2}a_{\frac{n}{2}}^* a_{\frac{n}{2}-1}^* \cdots a_1^*a_1\cdots a_{\frac{n}{2}-1} \right)^{\frac 12}\cdot \tau\left(a_n a_n^* a_{n-1}^* \cdots a_{\frac{n}{2}+1}^* a_{\frac{n}{2}+1}\cdots a_{n-1}\right)^{\frac12},\end{align*} where we used the Cauchy-Schwarz inequality in the first step, and condition (3) of the trace $\tau$ in the second. Now we apply the induction hypothesis for $n-2$ to the two terms of of the product. Indeed we have 

\begin{eqnarray*}\tau\left(a_\frac{n}{2}a_{\frac{n}{2}}^* a_{\frac{n}{2}-1}^* \cdots a_1^*a_1\cdots a_{\frac{n}{2}-1}\right)^{2\varphi(n-2)} 
&\leq &\ \tau\left( \left(a_{\frac{n}{2}}a_{\frac{n}{2}}^*a_{\frac{n}{2}}a_{\frac{n}{2}}^*\right)^{\varphi(n-2)}\right)\cdot\prod_{i=\frac{n}{2}-1}^{2}  \tau\left((a_ia_i^*)^{\varphi(n-2)}\right)  \\ &&\cdot \ \tau\left( (a_1^*a_1a_1^*a_1)^{\varphi(n-2)}\right)  \cdot\prod_{i=2}^{\frac{n}{2}-1}  \tau\left((a_i^*a_i)^{\varphi(n-2)}\right) \\ &=&  \tau\left( (a_1^*a_1)^{2\varphi(n-2)}\right)\cdot\prod_{i=2}^{\frac{n}{2}-1}  \tau\left((a_i^*a_i)^{\varphi(n-2)}\right)^2 \\ && \cdot \ \tau\left( \left(a_{\frac{n}{2}}^*a_{\frac{n}{2}}\right)^{2\varphi(n-2)}\right) \\ &\leq&  \prod_{i=1}^{\frac{n}{2}}  \tau\left((a_i^*a_i)^{2\varphi(n-2)}\right). 
\end{eqnarray*}

Note that for the first step we have used the induction hypothesis, for the second step property (3) of $\tau$ again, and in the third step the Cauchy-Schwarz inequality to see $$  \tau\left((a_i^*a_i)^{\varphi(n-2)}\right)^2\leq   \tau\left((a_i^*a_i)^{2\varphi(n-2)}\right).$$ The same procedure applies to the second term in the above inequality and all in all yields $$\vert \tau (a_1\cdots a_n)\vert \leq \prod_{i=1}^n \tau\left( (a_i^*a_i)^{2\varphi(n-2)}\right)^{\frac{1}{2\varphi(n-2)}\cdot\frac12}.$$ Since $2\varphi(n-2)=\varphi(n)$, the claim is proven for even $n$. For $n$ odd it now follows easily by setting $a_{n+1}=1$ and applying the already established result to $a_1,\ldots,a_{n+1}$.
\end{proof}

The result can be used to derive the following inequality, which will be used below: 

\begin{lemma}\label{sum} Let $\tau$ be a trace  on $A$. For all $r,n\in\N$ and all $a_1,\ldots,a_n\in A$ we have $$ \vert \tau\left( (a_1+\cdots +a_n)^r\right)\vert \leq \left( \sum_{i=1}^n \tau\left( (a_i^*a_i)^{\varphi(r)} \right)^{\frac{1}{2\varphi(r)}}\right)^r.$$ 
\end{lemma}
\begin{proof} Let $\mathcal{W}_n^r$ denote the set of all words in $a_1,\ldots,a_n$ of length exactly $r$. For $\omega\in\mathcal{W}_n^r$ denote by $\omega(i)$ the number of occurences of $a_i$ in $\omega$. We have \begin{align*} \vert \tau \left( (a_1+\cdots+a_n)^r\right)\vert &\leq \sum_{\omega\in \mathcal{W}_n^r} \vert \tau(\omega)\vert \\ &\leq \sum_{\omega\in\mathcal{W}_n^r} \prod_{i=1}^n \tau\left( (a_i^*a_i)^{\varphi(r)}\right)^{\frac{\omega(i)}{2\varphi(r)}} \\ &= \left( \sum_{i=1}^n \tau\left((a_i^*a_i)^{\varphi(r)}\right)^{\frac{1}{2\varphi(r)}} \right)^r.\end{align*} Note that we have used Lemma \ref{cs} for the second inequality.
\end{proof}

The following is a H\"older type inequality for $\Vert \cdot \Vert_p$:

\begin{proposition}\label{holder}
For any even $p$ and any $a,b\in A$ we have   $$\Vert ab\Vert_p\leq \Vert a\Vert_{\varphi(p+4)}\Vert b\Vert_{\varphi(p+4)}.$$

\end{proposition}
\begin{proof} We find
\begin{align*} \Vert ab\Vert_p & = \tau\left( (b^*a^*ab)^{\frac{p}{2}}\right)^{\frac1p} = \tau\left( a^*abb^* \cdots b^*a^*abb^*\right)^{\frac1p} \\ &\leq \tau\left( (a^*aa^*a)^{\varphi(p)}\right)^{\frac{1}{2p\varphi(p)}}   \tau\left( (bb^*bb^*)^{\varphi(p)}\right)^{\frac{1}{2p\varphi(p)}} \cdots \\ &= \left(\tau\left( (a^*aa^*a)^{\varphi(p)}\right)^{\frac{1}{2p\varphi(p)}}\right)^{\frac{p}{2}}  \left( \tau\left( (b^*bb^*b)^{\varphi(p)}\right)^{\frac{1}{2p\varphi(p)}} \right)^{\frac{p}{2}} \\ &= \tau\left( (a^*a)^{2\varphi(p)}\right)^{\frac{1}{4\varphi(p)}}   \tau\left( (b^*b)^{2\varphi(p)}\right)^{\frac{1}{4\varphi(p)}} \\ &= \Vert a\Vert_{\varphi(p+4)}\Vert b\Vert_{\varphi(p+4)},
\end{align*}
where we used property (3) of $\tau$ several times, and Lemma \ref{cs} for the inequality in the third step.
\end{proof}

We also get a Minkowski type inequality:

\begin{proposition} \label{mink} For any even $p$ and any $a,b\in A$ we have $$\Vert a+b \Vert_p \leq \Vert a \Vert_{\psi(p) }+ \Vert b\Vert_{\psi(p)},$$ where $\psi(p):= \varphi\left( 2\varphi(\frac{p}{2})+4 \right).$
\end{proposition}
\begin{proof} We use Proposition \ref{sum} to get 
\begin{eqnarray*}  \Vert a+b \Vert _p^2 &=& \tau\left(  (a^*a + a^*b + b^*a + b^*b)^{\frac{p}{2}}   \right)^{\frac2p} \\ &\leq&  \tau\left(  (a^*aa^*a)^{\varphi(\frac{p}{2})} \right)^{\frac{1}{2\varphi(\frac{p}{2})}} +  \tau\left(  (b^*aa^*b)^{\varphi(\frac{p}{2})} \right)^{\frac{1}{2\varphi(\frac{p}{2})}} \\ && + \quad \tau\left(  (a^*bb^*a)^{\varphi(\frac{p}{2})} \right)^{\frac{1}{2\varphi(\frac{p}{2})}}  + \tau\left(  (b^*bb^*b)^{\varphi(\frac{p}{2})} \right)^{\frac{1}{2\varphi(\frac{p}{2})}} \\ &=& \Vert a^*a\Vert_{2\varphi(\frac{p}{2})} + \Vert a^*b\Vert_{2\varphi(\frac{p}{2})}  + \Vert b^*a\Vert_{2\varphi(\frac{p}{2})} + \Vert b^*b\Vert_{2\varphi(\frac{p}{2})} .\end{eqnarray*}

If we now apply Proposition \ref{holder} to each term in this last sum we get $$  \Vert a+b \Vert _p^2 \leq \Vert a \Vert_{\psi(p)}^2 + 2\Vert a\Vert_{\psi(p)}\Vert b \Vert_{\psi(p)}+ \Vert b\Vert_{\psi(p)}^2= \left(\Vert a\Vert_{\psi(p)}+ \Vert b \Vert_{\psi(p)}\right)^2, $$ and from this the desired result.
\end{proof}

We can now give an easy proof of the following well-known Corollary:
\begin{corollary}\label{ideal1} Let $(A,\tau)$ be a tracial $*$-algebra. Then $$Z_{\tau}:=\{ a\in A\mid \Vert a\Vert_2=0 \}=\{a\in A\mid \Vert a\Vert_p = 0 \mbox{ for all even } p\}$$ is a two-sided $*$-ideal in $A$. The trace $\tau$ vanishes on $Z_{\tau}$ and thus induces a faithful trace on the $*$-algebra $A/Z_{\tau}$.
\end{corollary}
\begin{proof}
For any even $p$ and any $a\in A$ we have as a consequence of the usual Cauchy-Schwarz inequality \begin{align*} \Vert a\Vert_p^p = \tau( (a^*a)^{\frac{p}{2}}) \leq \tau (\underbrace{ a^*a\cdots a^*}_{p-1} \underbrace{ a^*a\cdots a^*}_{p-1})^{\frac12} \Vert a\Vert_2.\end{align*}  This shows that $Z_{\tau}$ is the set of elements with $\Vert a\Vert_p=0$ for all  even $p$, and with the above results $Z_{\tau}$ is then easily seen to be a two-sided $*$-ideal. From the usual Cauchy-Schwarz inequality we see that $\tau$ vanishes on $Z_{\tau}$. The rest is clear.
\end{proof}

\begin{definition} We call the algebra $A/Z_{\tau}$, with trace induced by $\tau$, the \textit{trace-reduction} of $(A,\tau)$.
\end{definition}

\begin{remark}\label{red}
It is clear that the trace-reduction $A/Z_{\tau}$ of a tracial $*$-algebra $(A,\tau)$ is trace-reduced.  Note that on a trace-reduced tracial $*$-algebra,  $\langle a,b\rangle := \tau(a^*b)$ defines an inner product. Thus $\Vert\cdot\Vert_2$ is a norm.

Let $\varphi\colon (A,\tau)\rightarrow (B,\rho)$ be a homomorphism of tracial $*$-algebras, i.e. a $*$-algebra homomorphism that fulfills $\rho(\varphi(a))=\tau(a)$ for all $a\in A$. Then $\Vert \varphi(a)\Vert_p=\Vert a\Vert_p$ holds for all $a\in A$. So $\ker\varphi\subseteq Z_{\tau},$ and if $(A,\tau)$ is  trace-reduced, then $\varphi$ is injective. 
Any $\varphi$ induces a tracial $*$-algebra embedding on the trace-reductions $A/Z_{\tau}\rightarrow B/Z_{\rho}.$
\end{remark}

Let $(A,\tau)$ be a trace-reduced tracial $*$-algebra.
It is clear that Theorem \ref{ineq} is telling us that the topology which is induced by the $p$-norms is compatible with addition and multiplication. Indeed, the sets
$$V_{n,p} := \left\{a \in A \mid \|a\|_p < \frac1n \right\}$$
form a countable subbasis of a vector space topology on $A$.
Hence, we obtain a metrizable topology on $A$ which endows $A$ with the structure of a metrizable topological algebra. We say that a sequence $(a_n)_{n \in \N}$ is a Cauchy sequence in $A$ if $\|a_n - a_m\|_p \to 0$ for all even $p$ as $n,m \to \infty$.
Moreover, a trace-reduced tracial $*$-algebra $A$ is \emph{complete} if every Cauchy sequence in $A$ has a limit in $A$.

It is easy to check that Theorem \ref{ineq} implies that every trace-reduced tracial $*$-algebra has a natural completion to a complete trace-reduced tracial $*$-algebra, which is defined as usual to be the space of Cauchy-sequences modulo null-sequences. 

\vspace{0.1cm}

We will also need the following observations:
\begin{lemma}\label{nil} A trace-reduced tracial $*$-algebra $A$ does not contain nilpotent left-ideals other than $\{0\}$.
\end{lemma}
\begin{proof} Let $I$ be a nilpotent left-ideal and $x\in I$. Then $x^*x\in I$, and thus $(x^*x)^{2^n}=0$ for some $n\geq 1$. So $\tau\left( (x^*x)^{2^n}\right)=0$, and since $\tau$ is faithful, this immediately implies $x=0$. 
\end{proof}

\begin{corollary} \label{embed}
Let $(A,\tau)$ be a trace-reduced tracial $*$-algebra which satisfies the polynomial identity $j_k$. Then $A$ embeds into a $k \times k$-matrix algebra over a commutative ring.\end{corollary}
\begin{proof} We already recalled the fact that the Jacobson radical of $A$ is nilpotent. By Lemma \ref{nil} it is trivial. The embedding result is now \cite{row}, Corollary 1.6.7. Concretely, there exists an embedding
$$\iota \colon A \hookrightarrow \prod_{m \subset A} A/m,$$
where $m$ runs through all maximal ideals. Indeed, the Jacobson ideal agrees with the intersection of all maximal ideals of $A$ and hence $\iota$ is injective. Now, $A/m$ is simple and satisfies the polynomial identity $j_k$. Hence, $A/m$ is isomorphic to a subring of the $k \times k$-matrices over a field $k_m$. We may consider the commutative $\C$-algebra $B=\prod_{m \subset A} k_m$ and see that
$A$ is now realized as a $*$-subalgebra of $M_k(B)$. This finishes the proof.
\end{proof}

\subsection{Metric ultraproducts of tracial $*$-algebras} 


We are now concerned with the notion of ultraproduct of tracial $*$-algebras. We will see that even though the definition is reasonable and natural, it has its pathologies. One of the unusual features is that the tracial ultraproduct of a sequence of tracial $*$-algebras may consist only of multiples of identity.

Let $I$ be an index set and $(A_i,\tau_i)$ a tracial $*$-algebra, for each $i\in I$. Most of the time, $I$ will be just the set $\N$. Define $$\widetilde{\prod}_{i\in I}^{} (A_i,\tau_i)=\left\{ (a_i)_{i\in I}\in\prod_{i\in I}A_i \mid  (\Vert a_i \Vert_p)_{i\in I} \mbox{ is bounded, for all even } p \right\}.$$ Here, $\Vert a_i \Vert_p= \tau_i\left( (a_i^*a_i)^{\frac{p}{2}}\right)^{\frac1p},$ as in the last section. Using the H\"older and Minkowski type inequalities from the last section, $\widetilde{\prod}_{i\in I}^{}(A_i,\tau_i)$ is easily seen to be a $*$-subalgebra of the product algebra.

Now let $\omega$ be an ultrafilter on $I$. For $x=(a_i)_{i\in I}\in \widetilde{\prod}_{i\in I}^{}(A_i,\tau_i)$ we define  $$\tau_\omega(x) := \lim_{i \to \omega} \tau_i(a_i). $$ Since $\vert\tau_i(a_i)\vert \leq \Vert a_i \Vert_2,$ and by the definition of $\widetilde{\prod}_{i\in I}^{}(A_i,\tau_i)$, this is a well defined mapping $$\tau_\omega\colon \widetilde{\prod}_{i\in I}^{}(A_i,\tau_i)\rightarrow \C.$$ One easily verifies that $\tau_\omega$ is even a trace. For the induced $p$-norms we find $$\Vert x\Vert _p = \tau_{\omega}\left( (x^*x)^{\frac{p}{2}}\right)^{\frac1p}= \lim_{i \to \omega} \tau_i\left( (a_i^*a_i)^{\frac{p}{2}}\right)^{\frac1p}=\lim_{i \to \omega} \Vert a_i \Vert_p.$$ So consider $$Z_{\tau_\omega}:=\left\{ x\in \widetilde{\prod}_{i\in I}^{}(A_i,\tau_i)\mid \Vert x\Vert_2 =0 \right\},$$ which is a two-sided $*$-ideal in $\widetilde{\prod}_{i\in I}^{}(A_i,\tau_i)$ by Corollary \ref{ideal1}.

\begin{definition} We call the quotient $*$-algebra $$\prod_{i\to \omega}(A_i,\tau_i):=\left(\widetilde{\prod}_{i\in I}^{}(A_i,\tau_i)\right)/ Z_{\tau_\omega}$$ the \textit{metric ultraproduct}  with respect to the ultrafilter $\omega$ of the family $\left((A_i,\tau_i)\right)_{i\in I}$. We equip it with the trace induced by $\tau_{\omega}$. 
\end{definition}

\begin{remark}\label{funkt}
The construction is functorial in the following sense: if $\varphi_i\colon (A_i,\tau_i)\rightarrow (B_i,\rho_i)$ are homomorphism of tracial $*$-algebras for all $i\in I,$ then they induce a canonical homomorphism of tracial $*$-algebras $\varphi\colon \prod_{i\to \omega} (A_i,\tau_i) \rightarrow \prod_{i\to \omega} (B_i,\rho_i).$ The usual functorial properties are fulfilled. $\varphi$ is injective, since $\prod_{i\to \omega} (A_i,\tau_i)$ is trace-reduced, see Remark \ref{red}. A sufficient condition for $\varphi$ being surjective is for example the following: $$\{i\in I\mid \varphi_i \mbox{ is surjective } \}\in \omega.$$ This is easily checked.
So for example replacing each $A_i$ by its trace-reduction $A_i/Z_{\tau_i}$ does not change the metric ultraproduct.
\end{remark}

\begin{lemma}
The metric ultraproduct of a sequence of tracial $*$-algebras is always trace-reduced and complete.
\end{lemma}
\begin{proof}
It is trace-reduced since we defined it to be the trace-reduction of a tracial $*$-algebra. A standard diagonalization argument shows that the metric ultraproduct is also complete.
\end{proof}

\begin{example}\label{emb} Let $A$ be a $*$-algebra and let $\tau, \tau_i$ for $i\in \N$ be traces on $A$. Assume $\lim_{i\to\infty} \tau_i =\tau$, pointwise on $A$. Consider the diagonal embedding $$\iota\colon A\rightarrow\widetilde{\prod}_{i\in I} (A,\tau_i); \quad a\mapsto (a)_{i\in I}.$$ Since $\tau_i((a^*a)^{\frac{p}{2}})\to \tau((a^*a)^{\frac{p}{2}})$ for fixed $p$ and $a$, each diagonal sequence is indeed bounded with respect to $\Vert\cdot\Vert_p$. Fix an ultrafilter $\omega$ on $\N$ which is not principal. Then $\iota$ is a tracial $*$-algebra homomorphism, when $A$ is equipped with $\tau$. Thus we get a morphism of tracial $*$-algebras $$(A,\tau) \rightarrow \prod_{i\to \omega} (A,\tau_i).$$
\end{example}

\section{Embedding theorems}
We can now state the approximation results from Theorems \ref{main1} in a more conceptual way, as embedding results.

\begin{theorem}  \label{main}
Let $(A,\tau)$ be a tracial $*$-algebra, generated by $n$ hermitean elements. Then there is a tracial $*$-algebra homomorphism into a metric ultraproduct $$\varphi\colon (A,\tau)\rightarrow \prod_{k\to \omega} (A_{n,k},\tau_k)$$ on the index set $\N$, where $A_{n,k}$ is the algebra of $n$ generic self-adjoint $k\times k$-matrices and $\tau_k \colon A_{n,k} \to \C$ is a complex-linear functional which makes $(A_{n,k},\tau_k)$ into a tracial $*$-algebra for each $k \in \N$.
\end{theorem}
\begin{proof} First note that we can assume that $A=\C\langle X\rangle$ for $X = \{X_1,\dots,X_n\}$. Indeed choose a surjective $*$-homomorphism $\pi\colon \C\langle X\rangle \rightarrow A$, and pull back the trace $\tau$ to a trace $\tau'$ on  $\C\langle X\rangle$. Then $\ker \pi\subseteq Z_{\tau'}$, see Remark \ref{red}. So we obtain a surjection $A \twoheadrightarrow \C\langle X\rangle/Z_{\tau'},$ and any trace-preserving homomorphism from $\C\langle X\rangle$ to an ultraproduct factors through $\C\langle X\rangle/Z_{\tau'}$, again by Remark \ref{red}.

But for $A=\C\langle X\rangle$ we have shown that $\tau$ can be approximated by traces $\tau_k$ that vanish on $J_k$, see Theorem \ref{main1}. So as described in Example \ref{emb} we get a tracial $*$-algebra homomorphism $$(\C\langle X\rangle,\tau)\rightarrow \prod_{k\to \omega} (\C\langle X\rangle, \tau_k)$$

But since $\tau_k$ vanishes on $J_k$ we can replace  $(\C\langle X\rangle,\tau_k)$ by $A_{n,k}=\C\langle X\rangle /J_k$ with the induced trace, without changing the ultraproduct (see Remark \ref{funkt}). This finishes the proof.
\end{proof}

\begin{theorem} \label{mainc}Let $(A,\tau)$ be a countably generated tracial $*$-algebra. Then there is a tracial $*$-algebra homomorphism into a metric ultraproduct $$\varphi\colon (A,\tau)\rightarrow \prod_{k\to \omega} (A_{k},\tau_k),$$ where $A_{k}$ is the algebra of generic self-adjoint $k\times k$-matrices.
\end{theorem}
\begin{proof} The proof is exactly the same as before, reducing to $A=\C\langle X_{\infty}\rangle$ first and applying Remark \ref{infset}. \end{proof}

\begin{corollary} \label{coroneumann}
Let $(A,\tau)$ be a trace-reduced tracial $*$-algebra and suppose that $A$ admits a countably generated dense subalgebra. Then there exist trace-reduced tracial $*$-algebras $(A_n,\tau_n)$ of type I$_{\leq n}$ and a trace-preserving embedding into a metric ultraproduct
$$\iota \colon (A,\tau) \hookrightarrow \prod_{n \to \omega} (A_n,\tau_n).$$
\end{corollary}

In order to apply the preceding corollary to finite von Neumann algebras with a separable pre-dual, we should convince ourselves that a finite von Neumann algebra with a separable pre-dual admits indeed a countably generated dense subalgebra with respect to the topology induced by all the $p$-norms. 

\begin{proposition} Let $(M,\tau)$ be a finite von Neumann algebra with a separable pre-dual and a specified trace $\tau \colon M \to \C$. Then, there exists a countably generated ultra-weakly dense $\C$-subalgebra
$A \subset M$, and for any such algebra $(A,\tau|_A)$, the completion of $(A,\tau|_p)$ with respect to the $p$-norms is canonically contained in the GNS Hilbert space $L^2(M,\tau)$ and contains $M$ as the sub-algebra of bounded elements.
\end{proposition}
\begin{proof}
It is well-known, that if $M$ has a separable pre-dual, then $M$ admits a countable ultra-weakly dense set and hence a countably generated $\C$-subalgebra $A$, which is dense in the ultra-weak topology. Since any Cauchy sequence with respect to the $p$-norms is a Cauchy sequence with respect to the 2-norm, we may identify the completion with a subspace of the GNS space $L^2(A,\tau)$.

We denote the norm-completion of $A$ by $C^*(A,\tau)$. It is clear that any norm-limit is also a limit with respect to the $p$-norms since
$$\|x\| = \sup_{p} \|x\|_p, \quad \forall x \in M.$$ 
Hence, $C^*(A,\tau)$ is contained in the completion with respect to the $p$-norms.
Pedersen's Theorem \cite[Thm. 2.7.3]{ped} states that for every $x \in M$ and $n \in \N$, there exists a projection $q_n \in M$ with $\tau(q_n) \geq 1 - \frac1n$ and $y_n \in C^*(A,\tau)$ with $\|y_n\| \leq \|x\|$, such that $\|q_n(x-y_n)\| \leq \frac1n$.

We conclude from this
$$\|x-y_n\|_p \leq \|q_n(x-y_n)\|_p + \|(1-q_n)(x- y_n)\|_p \leq \frac1n + \|1-q_n\|_p \|x-y_n\| \leq \frac1n + \frac1{n^{1/p}} \cdot 2\|x\|.$$
Here, we used $\|1-q_n\|_p = \tau((1-q_n)^p)^{1/p} \leq \frac1{n^{1/p}}$. In particular, we conclude that  $y_n \to x$ in the topology induced by the $p$-norms. Hence $M$ is contained in the completion and it is easy to see that $M$ consists precisely of those elements which are bounded with respect to the trace.
\end{proof}
\begin{remark} 
We note that a finite von Neumann algebra with a specified trace is not complete with respect to the $p$-norms unless it is finite dimensional.
\end{remark}
\section*{Acknowledgment}

The results in Section \ref{secapp} were obtained in 2007 by the second author in an unsuccessful attempt to prove Connes' embedding conjecture. He wants to thank Eberhard Kirchberg for pointing out the crucial omission of the boundedness issue. The occuring pathologies in GNS-representations with respect to unbounded traces are astonishing and in harsh contrast to the widely experienced automatic regularity of the analytic behavior in the presence of traces -- for example when working with the algebra of unbounded operators affiliated with a finite von Neumann algebra. Sufficient conditions for unbounded traces to lead to a well-behaved representation of a finite von Neumann algebra have been studied for example in \cite{MR722244}.
However, it is well-known in the context of Real Algebraic Geometry that positive traces on the commutative algebra $\C[X_1,X_2] = \C\langle X_1,X_2 \rangle / J_1$ can already be rather pathological and take negative values on polynomials which are strictly positive on the plane $\R^2$ but are not a sums of squares of polynomials. Hence, the GNS-representation of such functionals cannot yield a measure on $\R^2$ and must necessarily have its pathologies. It is also clear that such a trace cannot be the pointwise limit of bounded traces.

Over the years, the insight grew that the results in Section \ref{secapp}, and the consequences which we discussed in this paper, have indeed only very little to do with Connes' original conjecture.

The second author wants to thank Konrad Schm\"udgen for encouragement and helpful discussions about concepts of Non-commutative Real Algebraic Geometry \cite{schmue}.


\begin{bibdiv}
\begin{biblist}



\bib{amitlev1}{article}{
author={Amitsur, A. S.},
author={Levitzki, J.},
title={Minimal identities for algebras},
journal={Proc. Amer. Math. Soc.},
volume={1},
date={1950},
pages={449--463},
issn={0002-9939},
}

\bib{amitlev2}{article}{
author={Amitsur, A. S.},
author={Levitzki, J.},
title={Remarks on minimal identities for algebras},
journal={Proc. Amer. Math. Soc.},
volume={2},
date={1951},
pages={320--327},
issn={0002-9939},
}

\bib{ampro}{article}{
 AUTHOR = {Amitsur, A. S.}
 author= {Procesi, C.},
  TITLE = {Jacobson-rings and {H}ilbert algebras with polynomial
           identities},
JOURNAL = {Ann. Mat. Pura Appl. (4)},
FJOURNAL = {Annali di Matematica Pura ed Applicata. Serie Quarta},
 VOLUME = {71},
   YEAR = {1966},
  PAGES = {61--72},
   ISSN = {0003-4622},
MRCLASS = {16.49},
MRNUMBER = {MR0206044 (34 \#5869)},
MRREVIEWER = {D. S. Rim},
}


\bib{barvinok}{book}{
 AUTHOR = {Barvinok, A.},
  TITLE = {A course in convexity},
 SERIES = {Graduate Studies in Mathematics},
 VOLUME = {54},
PUBLISHER = {American Mathematical Society},
ADDRESS = {Providence, RI},
   YEAR = {2002},
  PAGES = {x+366},
   ISBN = {0-8218-2968-8},
MRCLASS = {52-02 (49N15 52-01 90-02 90C05 90C22 90C25)},
}


\bib{bisgaard}{article}{
 AUTHOR = {Bisgaard, T. M.},
  TITLE = {The topology of finitely open sets is not a vector space
           topology},
JOURNAL = {Arch. Math. (Basel)},
FJOURNAL = {Archiv der Mathematik},
 VOLUME = {60},
   YEAR = {1993},
 NUMBER = {6},
  PAGES = {546--552},
   ISSN = {0003-889X},
  CODEN = {ACVMAL},
MRCLASS = {46A99 (54H99)},
}

\bib{br}{article}{
 AUTHOR = {Braun, A.},
  TITLE = {The radical in a finitely generated {P}.{I}. algebra},
JOURNAL = {Bull. Amer. Math. Soc. (N.S.)},
FJOURNAL = {American Mathematical Society. Bulletin. New Series},
 VOLUME = {7},
   YEAR = {1982},
 NUMBER = {2},
  PAGES = {385--386},
   ISSN = {0273-0979},
  CODEN = {BAMOAD},
MRCLASS = {16A38 (16A21)},
}


\bib{connes1}{article}{
 author={Connes, A.},
 title={Classification of injective factors. Cases $II\sb{1},$
 $II\sb{\infty },$ $III\sb{\lambda },$ $\lambda \not=1$},
 journal={Ann. of Math. (2)},
 volume={104},
 date={1976},
 number={1},
 pages={73--115},
 issn={0003-486X},
}


\bib{fackkos}{article}{
author={Fack, T.},
author={Kosaki, H.},
title={Generalized $s$-numbers of $\tau$-measurable operators},
journal={Pacific J. Math.},
volume={123},
date={1986},
number={2},
pages={269--300},
issn={0030-8730},
}


\bib{hadwin}{article}{
 author={Hadwin, D.},
 title={A noncommutative moment problem},
 journal={Proc. Amer. Math. Soc.},
 volume={129},
 date={2001},
 number={6},
 pages={1785--1791 (electronic)},
 issn={0002-9939},
}


\bib{jac}{book}{
  AUTHOR={Jacobson, N.}
  TITLE = {{${\rm PI}$}-algebras},
 SERIES = {Lecture Notes in Mathematics, Vol. 441},
   NOTE = {An introduction},
PUBLISHER = {Springer-Verlag},
ADDRESS = {Berlin},
   YEAR = {1975},
  PAGES = {iv+115},
MRCLASS = {16A38},
MRNUMBER = {MR0369421 (51 \#5654)},
MRREVIEWER = {Edward Formanek},
}



\bib{klepsch}{article}{
 author={Klep, I.},
 author={Schweighofer, M.},
 title={Connes' embedding conjecture and sums of Hermitian squares},
 journal={Adv. Math.},
 volume={217},
 date={2008},
 number={4},
 pages={1816--1837},
 issn={0001-8708},
}
		


\bib{mvn1}{article}{
 author={Murray, F. J.},
 author={von Neumann, J.},
 title={On rings of operators},
 journal={Ann. of Math. (2)},
 volume={37},
 date={1936},
 number={1},
 pages={116--229},
 issn={0003-486X},
}

		
\bib{mvn2}{article}{
 author={Murray, F. J.},
 author={von Neumann, J.},
 title={On rings of operators. II},
 journal={Trans. Amer. Math. Soc.},
 volume={41},
 date={1937},
 number={2},
 pages={208--248},
 issn={0002-9947},
}
		
\bib{mvn3}{article}{
 author={Murray, F. J.},
 author={von Neumann, J.},
 title={On rings of operators. IV},
 journal={Ann. of Math. (2)},
 volume={44},
 date={1943},
 pages={716--808},
 issn={0003-486X},
}


\bib{pa}{article}{
  AUTHOR = {Paschke, W. L.},
   TITLE = {{$L\sp 2$}-homology over traced {$\ast$}-algebras},
 JOURNAL = {Trans. Amer. Math. Soc.},
FJOURNAL = {Transactions of the American Mathematical Society},
  VOLUME = {349},
    YEAR = {1997},
  NUMBER = {6},
   PAGES = {2229--2251},
    ISSN = {0002-9947},
   CODEN = {TAMTAM},
 MRCLASS = {46M20 (46K10)},
MRNUMBER = {MR1407708 (97j:46080)},
MRREVIEWER = {Gustavo Corach},
}


\bib{ped}{book}{
 author={Pedersen, G. K.},
 title={$C^{\ast} $-algebras and their automorphism groups},
 series={London Mathematical Society Monographs},
 volume={14},
 publisher={Academic Press Inc. [Harcourt Brace Jovanovich Publishers]},
 place={London},
 date={1979},
 pages={ix+416},
 isbn={0-12-549450-5},
}


\bib{pos}{article}{
 AUTHOR = {Posner, E. C.},
  TITLE = {Prime rings satisfying a polynomial identity},
JOURNAL = {Proc. Amer. Math. Soc.},
FJOURNAL = {Proceedings of the American Mathematical Society},
 VOLUME = {11},
   YEAR = {1960},
  PAGES = {180--183},
   ISSN = {0002-9939},
MRCLASS = {16.00},
}

\bib{row}{book}{
 AUTHOR = {Rowen, L. H.},
  TITLE = {Polynomial identities in ring theory},
 SERIES = {Pure and Applied Mathematics},
 VOLUME = {84},
PUBLISHER = {Academic Press Inc.},
ADDRESS = {New York},
   YEAR = {1980},
  PAGES = {xx+365},
   ISBN = {0-12-599850-3},
MRCLASS = {16A38},
}

\bib{radulescu}{article}{
 author={R{\u{a}}dulescu, F.},
 title={Convex sets associated with von Neumann algebras and Connes'
 approximate embedding problem},
 journal={Math. Res. Lett.},
 volume={6},
 date={1999},
 number={2},
 pages={229--236},
 issn={1073-2780},
}


\bib{schm1}{article}{
author={Schm{\"u}dgen, K.},
title={Graded and filtrated topological $\sp{\ast}$-algebras. II. The
closure of the positive cone},
journal={Rev. Roumaine Math. Pures Appl.},
volume={29},
date={1984},
number={1},
pages={89--96},
issn={0035-3965},
}


\bib{schm2}{book}{
author={Schm{\"u}dgen, K.},
title={Unbounded operator algebras and representation theory},
series={Operator Theory: Advances and Applications},
volume={37},
publisher={Birkh\"auser Verlag},
place={Basel},
date={1990},
pages={380},
isbn={3-7643-2321-3},
}


\bib{schmue}{article}{
 author={Schm{\"u}dgen, K.},
 title={Noncommutative real algebraic geometry---some basic concepts and
 first ideas},
 conference={
    title={Emerging applications of algebraic geometry},
 },
 book={
    series={IMA Vol. Math. Appl.},
    volume={149},
    publisher={Springer},
    place={New York},
 },
 date={2009},
 pages={325--350},

}


\bib{tak2}{book}{
 author={Takesaki, M.},
 title={Theory of operator algebras. II},
 series={Encyclopaedia of Mathematical Sciences},
 volume={125},
 note={;
 Operator Algebras and Non-commutative Geometry, 6},
 publisher={Springer-Verlag},
 place={Berlin},
 date={2003},
 pages={xxii+518},
 isbn={3-540-42914-X},
}



\bib{MR722244}{article}{
 author={Takesue, K.},
 title={Standard representations induced by positive linear functionals},
 journal={Mem. Fac. Sci. Kyushu Univ. Ser. A},
 volume={37},
 date={1983},
 number={2},
 pages={211--225},
 issn={0373-6385},
}


\end{biblist}
\end{bibdiv}
\end{document}